\documentclass[12pt]{article}
\usepackage{epsfig,psfrag,amsmath,amssymb,latexsym}
\usepackage{amscd}
\usepackage{color}
\usepackage{amsfonts}
\usepackage{graphicx}
\usepackage{wasysym}
\usepackage{mathrsfs}
\usepackage{verbatim}
\usepackage{hyperref}
\usepackage{cancel}
\usepackage[utf8]{inputenc}
\numberwithin{equation}{section}
\newtheorem{thm}{Theorem}[section]

\newtheorem{theorem}[thm]{Theorem}
\newtheorem{fact}[thm]{Fact}

\newtheorem{lemma}[thm]{Lemma}

\newtheorem{model}[thm]{Model}
\newenvironment{proof}[1][Proof]{\textbf{#1.} }{\ \rule{0.5em}{0.5em}}

\begin{document}
\thispagestyle{empty}

\begin{center}
  {\LARGE Transience of vertex-reinforced jump processes with long-range jumps}\\[3mm]
{\large Margherita Disertori\footnote{Institute for Applied Mathematics
\& Hausdorff Center for Mathematics, 
University of Bonn, \\
Endenicher Allee 60,
D-53115 Bonn, Germany.
E-mail: disertori@iam.uni-bonn.de}
\hspace{1cm} 
Franz Merkl \footnote{Mathematisches Institut, Ludwig-Maximilians-Universit{\"a}t  M{\"u}nchen,
Theresienstr.\ 39,
D-80333 Munich,
Germany.
E-mail: merkl@math.lmu.de
}
\hspace{1cm} 
Silke W.W.\ Rolles\footnote{
Department of Mathematics, CIT, 
Technische Universit{\"{a}}t M{\"{u}}nchen,
Boltzmannstr.\ 3,
D-85748 Garching bei M{\"{u}}nchen,
Germany.
E-mail: srolles@cit.tum.de}
\\[3mm]
{\small \today}}\\[3mm]
\end{center}

\newcommand{\htwo}{H^{2|2}}
\newcommand{\eps}{{\varepsilon}}
\newcommand{\Id}{\operatorname{Id}}
\newcommand{\diag}{\operatorname{diag}}
\newcommand{\supp}{\operatorname{supp}}
\newcommand{\range}{\operatorname{range}}
\newcommand{\sk}[1]{\left\langle{#1}\right\rangle}
\newcommand{\cI}{\mathcal{I}}
\newcommand{\cR}{\mathcal{R}}
\newcommand{\rr}{\mathcal{R}}
\newcommand{\W}{\mathcal{W}}
\newcommand{\E}{{\mathbb E}}
\newcommand{\R}{{\mathbb R}}  
\newcommand{\N}{{\mathbb N}}  
\newcommand{\Z}{{\mathbb Z}}
\newcommand{\cS}{{\mathcal{S}}}
\newcommand{\T}{{\mathcal{T}}}
\newcommand{\G}{{\mathcal{G}}}
\newcommand{\ceins}{c_H}
\newcommand{\cdrei}{{C}}
\newcommand{\weins}{{\overline W_1}}
\newcommand{\wzwei}{{\overline W_2}}

\begin{abstract}
  We show that the vertex-reinforced jump process on the $d$-dimensional lattice
  with long-range jumps is transient in any dimension $d$ as long as the initial
  weights do not decay too fast. 
  The main ingredients in the proof are:
  an analysis of the corresponding random
  environment on finite boxes, a comparison with a hierarchical model, and the
  reduction of the hierarchical model to a non-homogeneous effective
  one-dimensional model.
  For $d\ge 3$ we also prove transience of the vertex-reinforced jump process
  with possibly long-range jumps as long as nearest-neighbor weights are large enough.
  \footnote{\emph{MSC2020 subject classification:} primary 60K35, 60K37; secondary 60G60.}
\end{abstract}

\section{Introduction and results}

Consider an undirected connected locally finite
graph $G$ with vertex set $\Lambda$ and edge
set $E_\Lambda$. Every edge $e=\{i,j\}\in E_\Lambda$ is given a positive weight
$W_{ij}=W_{ji}=W_e>0$.
The vertex-reinforced jump process $(Y_t)_{t\ge 0}$  (VRJP for short) is a
stochastic process in continuous time which starts at time 0 in a vertex
$0\in\Lambda$ and conditionally on $(Y_s)_{s\le t}$ jumps from $Y_t=i$ along an
edge $\{i,j\}$ at rate
\begin{align}
W_{ij}\left(1+\int_0^t1_{\{Y_s=j\}}\, ds\right).
\end{align}
The process was conceived by Wendelin Werner in 2000
and first studied by Davis and Volkov \cite{davis-volkov1,davis-volkov2}
on trees and finite graphs. For a recent overview of the subject see
  \cite{bauerschmidt-helmuth-survey2021}.
It was shown by Sabot and Tarr\`es \cite{sabot-tarres2012}
and independently by Angel, Crawford, and Kozma \cite{angel-crawford-kozma}
that for any graph of bounded degree there exists $W_1>0$ such that
the VRJP is recurrent if $W_e\le W_1$ for
all $e\in E_\Lambda$. In contrast, Sabot and Tarr\`es \cite{sabot-tarres2012}
proved that for any $d\ge 3$, there exists $W_2<\infty$
such that the VRJP on $\Z^d$ is transient
if the weights $W_e$ are equal to a constant $\ge W_2$.
Sabot and Zeng in \cite{sabot-zeng15}
proved the following zero-one law for infinite
graphs. If the graph $G$ together with the weights $W=(W_e)_{e\in E_\Lambda}$ is
vertex transitive, the VRJP is either
almost surely recurrent or it is almost surely transient.
Poudevigne \cite{poudevigne22} proved some
monotonicity in the weights $W_e$ which implies that for $\Z^d$,
$d\ge 3$, there is a unique phase transition between recurrence and transience
when the weights $W_e$ are constant. 
On $\Z^d$ for $d=1,2$, the VRJP with arbitrary constant
weights is always recurrent. In one dimension, this was shown for
$W_e=1$ by Davis and Volkov \cite{davis-volkov1} and for general constant $W$
by Sabot and Tarr\`es \cite{sabot-tarres2012}.
Sabot \cite{sabot-polynomial-decay2021} proved the result in two dimensions. 
The discrete-time process associated to VRJP on any \emph{locally finite}
graph is given by the annealed law of a random walk in random conductances;
cf.\ \cite{sabot-tarres2012} for finite graphs and \cite{sabot-zeng15}
for locally finite infinite graphs.

We will consider here 
VRJP on $\Z^d$ with long-range interactions. This is not a
locally finite graph, but VRJP can still be well-defined
as long as $\sum_{e\ni i}W_e<\infty$ for all vertices
$i\in\Lambda$. If in addition $\sup_{i\in\Lambda}\sum_{e\ni i}W_e<\infty$ holds, then
VRJP jumps only finitely often up to any finite time. 
It seems that the existing results on
VRJP have not considered this generalized context. Therefore,
we decided not to rely on existing results on infinite graphs,
but rather on finite pieces of infinite graphs only.

Also on infinite, possibly not locally finite graphs $(\Lambda,E_\Lambda)$
up to the time that
the process leaves a given finite set $\Lambda_N\subset \Lambda$ of vertices,
the discrete-time process
associated to 
VRJP is a random walk in random conductances. The corresponding random
conductance for any edge $e=\{i,j\}$ may depend on $\Lambda_N$
and can be encoded in the form $c_e=W_{ij}e^{u_i+u_j}$
with a random environment $u=(u_i)_{i\in\Lambda_N}$.
The law of $u$ is explicitly given in \cite[Theorem 2]{sabot-tarres2012}.
We consider wired boundary conditions obtained
by adding a ``wiring point'' $\rho$ to $\Lambda_N$, setting $u_\rho:=0$,
and connecting any vertex $i\in\Lambda_N$ to $\rho$
with the ``wiring weight'',
synonym ``pinning strength'',
\begin{align}
  \label{eq:def-pinning-general}
  h_i=W_{i\rho}=\sum_{\substack{j\in \Lambda\setminus\Lambda_N:\\\{i,j\}\in E_\Lambda}}W_{ij},
\end{align}
provided that this sum is finite.
The corresponding ``pinning conductances'' are given by
$c_{i\rho}=W_{i\rho}e^{u_i+u_\rho}=h_i e^{u_i}$.
Leaving $\Lambda_N$ corresponds to jumping to the wiring point $\rho$.
Let $E_{\Lambda_N}=\{e\in E_\Lambda:\; e\subseteq\Lambda_N\}$ denote the edge set
restricted to $\Lambda_N$. The edge set on the extended graph with vertex set
$\Lambda_{N}\cup \{\rho  \}$ is 
\begin{align}
\label{eq:edge-set-wired-bc}
  E_N=E_{\Lambda_N}\cup \{\{i,\rho\}:\; i\in\Lambda_N
  \text{ and }\{i,j\}\in E_\Lambda\text{ for some } j\in\Lambda\setminus\Lambda_N\}. 
\end{align}
With these notations, the law $\nu^{\Lambda_N}_{W,h}$
of the random environment $u$ for the VRJP starting in the wiring point
$\rho$ is described by
\begin{align}
  \label{eq:u-marginal}
  \nu^{\Lambda_{N}}_{W,h}(du)=
  e^{-\hspace{-0,3cm}\sum\limits_{\{i,j\}\in E_{\Lambda_{N}}}\hspace{-0,5cm}W_{ij}(\cosh(u_i-u_j)-1)-\sum\limits_{i\in\Lambda_{N}}\hspace{-0,2cm}h_i
  (\cosh u_i-1)}
  \hspace{-0,2cm}\sqrt{\sum_{T\in \mathcal{S}}\prod_{\{i,j \}\in T} W_{ij}e^{u_{i}+u_{j}}}
  \prod_{i\in \Lambda_{N}}\frac{e^{-u_i}}{\sqrt{2\pi}}\, du_i,
\end{align}
where   
$\mathcal{S}$ is the set of spanning trees over the graph
$(\Lambda_{N}\cup\{\rho\},E_N)$. 
Let $\E^{\Lambda_N}_{W,h}$ denote the expectation corresponding to $\nu^{\Lambda_{N}}_{W,h}$. 

\subsection{Results on the VRJP with long-range weights in
  any dimension $d\ge 1$}

In this subsection, we consider the following model.

\begin{model}[$d$-dimensional lattice with Euclidean long-range interactions]
\label{model:eucl}\mbox{}\\
Fix a dimension $d\in\N$ and two parameters $\overline W>0$ and $\alpha>1$.
We consider the $d$-dimensional lattice $\Z^d$ with edge weights 
$W_{ij}=W_{ji}=w(\|i-j\|_\infty)$ for $i,j\in\Z^d$, $i\neq j$, where the function 
$w:[1,\infty)\to (0,\infty)$ is monotonically decreasing and satisfies the
lower bound 
\begin{align}
  \label{eq:cond1-w}
 w(x)\ge \overline W \frac{(\log_2 x)^\alpha}{x^{2d}}
    \text{ for all }x\ge 1
\end{align}
and the summability condition
\begin{align}
  \label{eq:summability-long-range}
  \sum_{i\in\Z^d\setminus\{0\}}w(\|i\|_\infty)<\infty. 
\end{align}
\end{model}

Note that large values of the weights $W_{ij}$ correspond to weak reinforcement,
because jumps typically occur after a short time; hence the local time
cannot increase a lot before any typical jump.
Thus, the lower bound \eqref{eq:cond1-w} can be interpreted as a regime of weak reinforcement. 
We prove the following transience result.

\begin{theorem}[Transience of long-range VRJP]
  \label{thm:vrjp-transient-long-range}
The discrete time process associated to the VRJP on
$\Z^d$ started in $0$ with interactions as in Model~\ref{model:eucl}
is transient, i.e.\ it visits a.s.\ any given vertex only finitely often.
Slightly stronger, the expected number of visits to any given vertex $x\in\Z^d$ 
is finite and uniformly bounded in $x$.
\end{theorem}

\paragraph*{Discussion}
We believe that a logarithmic correction in \eqref{eq:cond1-w} in dimensions $d=1,2$ is crucial
for this theorem to be valid and not only for its proof.
Indeed, for locally finite graphs, Poudevigne's Theorem~4 in \cite{poudevigne22} shows
that the VRJP with weights $W$ is recurrent whenever the
corresponding random walk in the \emph{deterministic} conductances $W$ is recurrent. 
In dimensions $d=1,2$, the random walk in conductances $W_{ij}\le \overline W\|i-j\|_\infty^{-2d}$ for some $\overline W>0$
is known to be recurrent, see \cite{caputo-faggionato-gaudilliere2009} and
\cite{baeumler-2022}. Although $\Z^d$
with long-range interactions is not a locally finite graph, one may
conjecture that Poudevigne's result could possibly be extended to this case. 
This would imply recurrence of the corresponding VRJP
for $W_{ij}\le \overline W\|i-j\|_\infty^{-2d}$.

\paragraph*{Strategy of proof}
The proof of the transience result Theorem \ref{thm:vrjp-transient-long-range}
relies on an analysis of random walks
  in random conductances given in Theorem \ref{le:random-conductances-trans} below. It requires the
  transience of the Markovian random walk on $\Z^d$ endowed
  with the  \emph{deterministic} weights $W_{ij}$ from Model~\ref{model:eucl}.
It is known that for   $W_{ij}\ge \overline W\|i-j\|_\infty^{-s}$, $i,j\in\Z^d$, with $\overline W>0$, $s<2d$,
the Markovian random walk on the weighted graph is transient. This was shown
by Caputo, Faggionato, and
Gaudilli\`ere \cite[Appendix B.1]{caputo-faggionato-gaudilliere2009}
using characteristic functions and by  Bäumler in \cite{baeumler-2022} using
electrical networks. In contrast to this, we have a logarithmic correction
$W_{ij}\ge \overline W(\log\|i-j\|_\infty)^\alpha\|i-j\|_\infty^{-2d}$ with $\alpha>1$
in Model~\ref{model:eucl}.
Bäumler's method still applies in this case and is used to
prove the following lemma.

\begin{lemma}[Transience of long-range Markovian random walk]
  \label{le:baeumler}
  \mbox{}\\
  Let $w:[1,\infty)\to (0,\infty)$ be a monotonically decreasing
  function satisfying \eqref{eq:cond1-w} and \eqref{eq:summability-long-range} for some $d\in\N$, $\overline W>0$, and $\alpha>1$.
  Then, the Markovian random walk on $\Z^d$ endowed with the long-range
  deterministic weights $W_{ij}=w(\|i-j\|_\infty)$ for all $i,j\in\Z^d$
is transient. 
\end{lemma}

To study the VRJP on $\Z^d$ with weights from Model~\ref{model:eucl}, we analyze the process
on a sequence of finite boxes $\Lambda_{N}\subset \Z^{d}$  
with wired boundary conditions, given by the pinning
  \begin{align}
    \label{eq:pinning-wired-bc}
    0\leq h_i=\sum_{j\in\Z^d\setminus\Lambda_N}W_{ij}<\infty,\quad i\in\Lambda_N. 
  \end{align}
  Note that the last formula is a special case of
  formula~\eqref{eq:def-pinning-general}. 
The following bounds on the environment are the key ingredient of our
transience proof. 
They may be interpreted as long-range order in any dimension $d\ge 1$, provided
that the long-range weights do not decay too fast.

\begin{theorem}[Bounds on the environment -- long-range case]
  \label{thm:est-cosh-euclidean}
  For $N\in\N$, we consider the box $\Lambda_N:=\{0,1,\ldots,2^N-1\}^d\subseteq\Z^d$,
  $d\ge 1$, 
with the weights inherited from Model~\ref{model:eucl} with
parameters $\overline W>0$ and $\alpha>1$.
Then, for all $m\ge 1$ and $\sigma\in\{\pm 1\}$, the bound 
  \begin{align}
    \label{eq:eucl-est-cosh-ui}
    \E^{\Lambda_N}_{W,h}\left[e^{\sigma m u_i}\right]\le\cdrei
  \end{align}
holds with some constant 
$\cdrei=\cdrei(\overline W,d,\alpha,m)>0$, 
  uniformly in $i\in\Lambda_N$ and $N$.
\end{theorem}
An explicit expression for $\cdrei$ is given in formula \eqref{eq:def-cdrei}.

\subsection{Results on the VRJP with possibly long-range weights in dimension $d\ge 3$}
In this subsection, we consider the following model.

\begin{model}\textbf{\emph{($d$-dimensional lattice with large interactions in high dimension $d\ge 3$)}}
  \label{model:high-d}
Fix a dimension $d\ge 3$ and $\overline W>0$.
We consider the $d$-dimensional lattice $\Z^d$ with possibly long-range edge weights 
$W_{ij}=W_{ji}\geq 0$ for $i,j\in\Z^d$ with $i\neq j$,
which are lower bounded by weights of simple nearest-neighbor
random walk on $\Z^d$
\begin{align}
W_{ij}\geq \overline W 1_{\{ \|i-j\|_{2}=1\}}
\end{align}
and fulfill the uniform summability condition
\begin{equation}\label{eq:cond2-w}
\sup_{i\in \Z^{d}} \sum_{j\in \Z^{d}\setminus \{i \}} W_{ij}<\infty .
\end{equation}
\end{model}

\begin{theorem}[Transience of VRJP in high dimension]
  \label{thm:vrjp-transient-high-d}
  Consider Model~\ref{model:high-d} with parameters $d\ge 3$ and $\overline W$.
  There is a constant $\weins=\weins(d)>0$ such that the assumption
  $\overline W\ge\weins$ implies that the 
discrete time process associated to the VRJP starting in $0$ for this model 
is transient, i.e.\ it visits a.s.\ any given vertex only finitely often.
Slightly stronger, the expected number of visits to any given vertex $x$ is finite.
Moreover, if the weights $W_{ij}$ are translation invariant
this expectation is uniformly bounded in $x$.
One can take $\weins = \max\{\wzwei,10^8\}$ with $\wzwei$ as in Theorem
\ref{th:est-cosh-euclidean-highdim}.
\end{theorem}

Note that on locally finite graphs  the probability that the VRJP is transient is an increasing function
of the weights $W_{ij}$ (cf.\ \cite[Theorem~1]{poudevigne22}).
If this result was applicable also for non locally finite graphs, then 
transience of VRJP as stated in Theorem~\ref{thm:vrjp-transient-high-d} 
would follow from  the transience result on  $\Z^d$ for $d\ge 3$ with nearest-neighbor jumps
proved in \cite{sabot-tarres2012}. However, one may conjecture that this obstruction is only
a technical issue, not a fundamental failure of the approach.

For $\Z^d$ in dimension $d\ge 3$, we deduce the following bounds
for the environment. As in the long-range case, these bounds are the key ingredient
to prove transience.
  
\begin{theorem}[Bounds on the environment -- high dimension]\label{th:est-cosh-euclidean-highdim}
We consider any  finite box  $\Lambda_N\subset\Z^d$  with the weights inherited from Model
\ref{model:high-d}. There exists $\wzwei\geq 2^8$ such that for all
$\overline W\geq \wzwei$, all $m\in [1,\frac12\overline W^{\frac{1}{8}}]$,
and $\sigma\in\{\pm 1\}$, one has
\begin{align}
    \label{eq:eucl-est-cosh-ui2}
  \E^{\Lambda_N}_{W,h}[e^{\sigma m u_i}]\le 2^{2m+1}
  \end{align}
  uniformly in $i\in\Lambda_N$ and the box $\Lambda_{N}$. 
\end{theorem}
\noindent Note that in the case of a strict inequality $\overline W>\wzwei$,
we can take $m>1$,
which yields uniform integrability of $e^{u_i}$. This is formulated
in the following Lemma \ref{le:uniform-integrability}.

\subsection{Uniform integrability}

A potential strategy to prove Theorems~\ref{thm:vrjp-transient-long-range} and
\ref{thm:vrjp-transient-high-d}
uses a recurrence/transience criterion given in
\cite[Theorem~1]{sabot-zeng15} for locally finite graphs. Take a
sequence of finite sets $\Lambda_N\uparrow\Z^d$ with $0\in\Lambda_N$.
One can couple $u_0^{(N)}$ distributed according to
  $\nu^{\Lambda_N}_{W,h}$ for all $N$ in such a way that $(e^{u_0^{(N)}})_{N\in\N}$ is a
  non-negative martingale, which is uniformly integrable by the following
  lemma.

\begin{lemma}[Uniform integrability]
  \label{le:uniform-integrability}
  In Model~\ref{model:eucl}, let $\overline W>0$.
    In Model~\ref{model:high-d}, let $\overline W>\wzwei$ with $\wzwei$ from Theorem \ref{th:est-cosh-euclidean-highdim}.
  Under these assumptions, the random variables
  $e^{\pm u_i}$ distributed according to $\nu^{\Lambda_N}_{W,h}$
  are uniformly integrable in $N\in\N$ and $i\in\Lambda_N$.
\end{lemma}
  Hence, the martingale $(e^{u_0^{(N)}})_{N\in\N}$ converges a.s.\ and in $L^1$
  to a limit $\psi_0$. The limit satisfies 
  $\E[\psi_0]=\lim_{N\to\infty}\E^{\Lambda_N}_{W,h}[e^{u_0^{(N)}}]=1$, 
  where the last equation follows by supersymmetry, see
  \cite[(B.3)]{disertori-spencer-zirnbauer2010}.
  Because of uniform integrability of $e^{-u_0^{(N)}}$, $N\in\N$,
  the expectation $\E[\psi_0^{-1}]$ is finite and consequently
  the limit $\psi_0$ is almost surely strictly positive. If one extended
  \cite[Theorem~1]{sabot-zeng15} to $\Z^d$ with long-range interactions,
  this would yield transience for the VRJP in an alternative way to the approach
  in this paper.


  \paragraph*{How this paper is organized}
In Section \ref{sec:prooftrans}  we show how transience of the VRJP follows from the bounds on the
random environment given in Theorems \ref{thm:est-cosh-euclidean} and
\ref{th:est-cosh-euclidean-highdim}. Uniform integrability from
Lemma~\ref{le:uniform-integrability} is shown in Section
\ref{subsec:uniform}.
In Section  \ref{se:boundslongrange} we analyse an auxiliary model with hierarchical
interactions. The corresponding bounds are obtained 
by reducing the model to an effective one with vertices that represent length scales.
  Although this effective model is still defined on a complete graph, only interactions between
  neighboring length scales are relevant, which allows us to treat it essentially like
  an inhomogeneous one-dimensional model.
In Section \ref{se:long-range-weights}, we show how
bounds on the corresponding environment can be transferred to bounds on the environment for 
the Euclidean lattice in any dimension $d\ge 1$ with long-range interactions.
In Section~\ref{subsec:H22} we introduce the $H^{2|2}$ model, which describes the VRJP random
environment as a marginal. Section \ref{se:ward-identity} collects
some identities generated by supersymmetry. These are applied to deduce a bound
on fluctuations of the random environment.
Finally, in Section  \ref{se:boundshighd}, we prove the bounds on the environment
in dimension $d\geq 3$ from Theorem \ref{th:est-cosh-euclidean-highdim}
by combining the bound from Section \ref{se:ward-identity}  
with estimates from
\cite{disertori-spencer-zirnbauer2010} and a monotonicity result from \cite{poudevigne22}. 

The constants $\weins,\wzwei$, and $\cdrei$ keep their meaning throughout the whole
paper.

\section{Proof of the main results}
\subsection{Transience proof given bounds on the environment}
\label{sec:prooftrans}

Our transience proof uses the connection between random walks and electrical
networks, cf.\ \cite{lyons-peres2016} for background. First
we prove transience for Markovian random walks with long-range jumps. 

\begin{proof}[Proof of Lemma~\ref{le:baeumler}]
  By \cite[Theorem~2.11]{lyons-peres2016}, 
  to prove transience, it suffices to construct a unit flow of finite
  energy from $0$ to infinity.
  Following the same strategy as in the proof of 
  \cite[Theorem~1.1]{baeumler-2022},
  we consider the pairwise disjoint annuli
  $B_k=((-2^k,2^k]^d\setminus(-2^{k-1},2^{k-1}]^d)\cap\Z^d$
  for $k\in\N$ and $B_0=\{0\}^d$.  We observe 
  $|B_k|\ge 2^{kd}$ for $k\in\N_0$. Note that for
  $i\in B_k$, $j\in B_{k+1}$, we have $\|i-j\|_\infty\le 2^{k+2}$ and since
  $w$ is decreasing and satisfies the lower bound \eqref{eq:cond1-w},
  it holds 
  \begin{align}
    W_{ij}=w(\|i-j\|_\infty)\ge w(2^{k+2})
    \ge \overline W \frac{(k+2)^\alpha}{2^{2d(k+2)}}
    \ge \overline W 2^{-3d}\frac{(k+2)^\alpha}{|B_k||B_{k+1}|}
    .
  \end{align}
  Given $i,j\in\Z^d$, we define
  $\theta(i,j)=-\theta(j,i)=|B_k|^{-1}|B_{k+1}|^{-1}$ if $i\in B_k$, $j\in B_{k+1}$ for some
  $k\in\N_0$, and $\theta(i,j)=0$ otherwise. Then, $\theta$ is a unit flow
  from $0$ to infinity, i.e., Kirchhoff's node rule
  $\sum_{j\in\Z^d}\theta(i,j)=\delta_{0i}$ holds for all
  $i\in\Z^d$. The corresponding energy of the network with conductances
  $W$ is bounded as follows
  \begin{align}
    \frac12\sum_{i,j\in\Z^d}\frac{\theta(i,j)^2}{W_{ij}}
    =\sum_{k\in\N_0}\sum_{i\in B_k}\sum_{j\in B_{k+1}}\frac{|B_k|^{-2}|B_{k+1}|^{-2}}{W_{ij}}
    \le\frac{2^{3d}}{\overline W}\sum_{k\in\N_0}
    (k+2)^{-\alpha}<\infty
  \end{align}
  because $\alpha>1$.
\end{proof}

To prove transience for VRJP, we apply the following criterion
for random walks in random conductances. 
The idea for its proof is due to Christophe Sabot
and Pierre Tarr\`es (proof of \cite[Cor.\ 4]{sabot-tarres2012} and 
private communication).

\begin{theorem}[Transience criterion]
  \label{le:random-conductances-trans}
  Consider an undirected connected graph $G=(\Lambda,E_\Lambda)$, not
  necessarily of finite degree, a vertex $0\in\Lambda$,
  and an increasing sequence
  $\Lambda_N\uparrow\Lambda$ of finite vertex sets containing $0$. Let $G_N=(\Lambda_N\cup\{\rho\}, E_N)$,
  where $\rho$ is an additional wiring point and $G_N$ is obtained from $G$ as follows.
  We restrict the vertex set to $\Lambda_N$, add $\rho$, and use wired boundary conditions 
  to define the edge set $E_N$ as in \eqref{eq:edge-set-wired-bc}. 
  
  Consider a stochastic process $X=(X_n)_{n\in\N_0}$ on $G=(\Lambda,E_\Lambda)$ starting
  in $0$ and taking
  only nearest-neighbor jumps. We assume that for each $N$, the law of the process before
  it hits
  $\Lambda\setminus\Lambda_N$ equals the annealed law of a random walk in some random conductances
  $c_e=c_e^{(N)}>0$, $e\in E_N$,
  on $G_N$ before it hits $\rho$. The law of $(c_e^{(N)})_{e\in E_N}$ may depend on $N$.
  We denote the expectation averaging over the conductances
  by $\E^N$. Let $x\in\Lambda$. 
  Assume that there are deterministic conductances $W_e>0$, $e\in E_\Lambda$, such
  that the following hold: 
  \begin{enumerate}
  \item For all $i\in\Lambda$, one has $\sum_{e\in E_\Lambda:i\in e}W_e<\infty$. 
  \item The random walk on the weighted graph $(G,W)$ is transient.
  \item There is a constant $K_x>0$ such that
    for all $N\in\N$ and all $e\in E_N$
    \begin{align}
      \label{eq:bound-sum-cf-over-ce}
      \sum_{f\in E_N:\, x\in f}\E^N\left[\frac{c_f}{c_e}\right]\le \frac{K_x}{W_e},
    \end{align}
    where $W_{i\rho}=W_{i\rho}^{(N)}$ is defined via the wired boundary conditions as
    in \eqref{eq:def-pinning-general}. 
  \end{enumerate}
  Let $\rr^N(W,x\leftrightarrow\rho)$ denote the effective resistance between
  $x$ and $\rho$ in the network $G_N$ with deterministic conductances $W=(W_e)_{e\in E_N}$. 
  Then, the expected number of
  visits in $x$ of $X$ before exiting $\Lambda_N$ is
  bounded by $K_x\rr^N(W,x\leftrightarrow\rho)$, which is bounded 
  uniformly in $N$. As a consequence, the expected number of visits of $x$ by
  the process $X$ on $G$ is finite with the same bound. 
\end{theorem}
\begin{proof}
  Let $E_{\Lambda,0}^{\mathrm{rw}}$ denote the expectation with respect to the process $X$,
  which starts in $0$, on $G$. For any $y\in\Lambda_N$, 
  let $P_{c,y}^N$ and $E^N_{c,y}$ be the probability measure and expectation,
  respectively, underlying the Markovian
  random walk $X$ starting in $y$ on $G_N$ in given conductances~$c$.
  Let $\tau_x^+$ denote the first return
  time to $x$, and $\tau_x$, $\tau_\rho$ the 
    hitting time of $x$, $\rho$, respectively. Let $N_x:=\sum_{n=0}^\infty 1_{\{X_n=x,n<\tau_\rho\}}$
  denote the number of visits to $x$ before visiting $\rho$.
  The expected number of visits of $X$ to $x$ equals
  \begin{align}
    \label{eq:expected-number-of-visits}
    E_{\Lambda,0}^{\mathrm{rw}}\left[\sum_{n=0}^\infty 1_{\{X_n=x\}}\right]
    =&\lim_{N\to\infty}\E^N\left[E_{c,0}^N[N_x]\right].
  \end{align}
  Let $N$ be large enough that $x\in\Lambda_N$.
  Conditionally on $\tau_x<\tau_\rho$, since $c_e>0$ for all $e\in E_N$, the number of visits
    to $x$ before hitting $\rho$ in
  fixed conductances $c$ is geometric with mean 
  \begin{align}
    E_{c,x}^N[N_x]=
P_{c,x}^N(\tau_\rho<\tau_x^+)^{-1}
  =c(x)\rr^N(c,x\leftrightarrow\rho),
\end{align}
where $c(x)=\sum_{f\in E_N:x\in f}c_f$ and $\rr^N(c,x\leftrightarrow\rho)$ denotes
the effective resistance between $x$ and $\rho$ in the network $G_N$ with conductances
$c$; the expression of the escape probability in terms of the effective resistance
is described in \cite[formula~(2.5)]{lyons-peres2016}.
We obtain
\begin{align}
  \label{eq:bound-exp-N-x}
    E_{c,0}^N[N_x]=
    P_{c,0}^N(\tau_x<\tau_\rho)E_{c,x}^N[N_x]
    \le c(x)\rr^N(c,x\leftrightarrow\rho). 
\end{align}
    By  \cite[formula~(2.7)]{lyons-peres2016},
there is a unit flow $\theta=(\theta(i,j))_{i,j\in\Lambda_N:\{i,j\}\in E_N}$ from
$x$ to $\rho$ in this network, i.e.\ $\theta(i,j)=-\theta(j,i)$,
$\sum_{j\in\Lambda_N:\{i,j\}\in E_N}\theta(i,j)=\delta_{ix}-\delta_{i\rho}$,  
such that the effective resistance in the network $G_N$ with deterministic
conductances $W$ is given by
\begin{align}
  \rr^N(W,x\leftrightarrow\rho)
      =\sum_{\{i,j\}\in E_N}W_{ij}^{-1}\theta(i,j)^2. 
\end{align}
By Thomson's principle \cite[Section 2.4]{lyons-peres2016},
with the same flow $\theta$, we obtain an upper bound for the effective resistance
with random conductances
\begin{align}
  \rr^N(c,x\leftrightarrow\rho)\le \sum_{\{i,j\}\in E_N}c_{ij}^{-1}\theta(i,j)^2. 
\end{align}
Using first \eqref{eq:bound-exp-N-x} and then assumption \eqref{eq:bound-sum-cf-over-ce}, we obtain the following bound
for the expected number of visits in $x$ before exiting $\Lambda_N$
\begin{align}
   \E^N&\left[E_{c,0}^N[N_x]\right]
  \le\E^N[c(x)\rr^N(c,x\leftrightarrow\rho)] \nonumber\\
       &\le \sum_{\{i,j\}\in E_N}\sum_{\substack{f\in E_N:\\ x\in f}} \E^N\left[\frac{c_f}{c_{ij}}\right]\theta(i,j)^2
  \le K_x\sum_{\{i,j\}\in E_N}W_{ij}^{-1}\theta(i,j)^2
  \nonumber\\
  &=K_x\rr^N(W,x\leftrightarrow\rho)
       \uparrow_{N\to\infty}K_x\rr(W,x\leftrightarrow\infty)
       \label{eq:upper-bound-resistance}
\end{align}
with a finite limit $\rr(W,x\leftrightarrow\infty)$
because of the transience assumption. In particular, the estimate
in \eqref{eq:upper-bound-resistance} is uniform in $N$. 
Finally, \eqref{eq:expected-number-of-visits} allows us to conclude.  
\end{proof}

Using the bounds in Theorems \ref{thm:est-cosh-euclidean} and
\ref{th:est-cosh-euclidean-highdim},
we apply the criterion from Theorem~\ref{le:random-conductances-trans} to
deduce transience of VRJP. 

\begin{proof}[Proof of Theorem~\ref{thm:vrjp-transient-long-range}]
  We use Model \ref{model:eucl} on $\Z^d$ with the weights $W_{ij}$, $i,j\in\Z^d$, and
  the boxes $\Lambda_N=\{0,\ldots,2^N-1\}^d$ replaced by shifted boxes
    $\tilde\Lambda_N=\Lambda_N+(a_N,\ldots,a_N)$ with $a_N\in\Z$ suitably chosen
    such that $\tilde\Lambda_N\uparrow\Z^d$ and   $0\in\tilde\Lambda_N$.
    We take the graphs $G_N=(\tilde\Lambda_N\cup\{\rho\},E_N)$ with wired boundary
    conditions as specified in Theorem~\ref{le:random-conductances-trans}.
    Let $W_{i\rho}=W_{i\rho}^{(N)}$ be as in Formula \eqref{eq:def-pinning-general}.
  By \cite[Theorem~2]{sabot-tarres2012}, the discrete-time process associated with
  the VRJP on $G_N$ starting at $\rho$ has the
  same distribution as the random walk in random conductances $c_{ij}=W_{ij}e^{u_i+u_j}$,
  $i,j\in\tilde\Lambda_N\cup\{\rho\}$, on $G_N$, where $u_i$, $i\in\tilde\Lambda_N$, are
  jointly distributed according to the law $\nu^{\tilde\Lambda_N}_{W,h}$ given in
  \eqref{eq:u-marginal}.
  Changing the starting point from $\rho$ to 0, one may take the same
  random conductances $W_{ij}e^{u_i+u_j}$, but with respect to the modified
  measure $e^{u_0}\, d\nu^{\tilde\Lambda_N}_{W,h}$ having the density
  $e^{u_0-u_\rho}=e^{u_0}$ with respect to $\nu^{\tilde\Lambda_N}_{W,h}$. This can be seen from
  formula~\eqref{eq:u-marginal} using the shift $\tilde u_i=u_i-u_0$
  and noting that a factor $e^{-u_0}$ is replaced by $e^{-u_\rho}$
  when changing the product
  $\prod_{i\in\Lambda_N}e^{-u_i}$ to $\prod_{i\in\Lambda_N\cup\{\rho\}\setminus\{0\}}e^{-u_i}$. 

  We verify the assumptions 1-3 of Theorem~\ref{le:random-conductances-trans}.
    Assumption 1 holds by hypothesis \eqref{eq:summability-long-range}. 
    Assumption 2 is true by Lemma~\ref{le:baeumler}. 
  Finally, Assumption 3 is verified as follows.
  Let $N\in\N$. 
  The expectation $\E^N$ in Theorem~\ref{le:random-conductances-trans} is with
  respect to the environment for the VRJP starting
  in $0$, not in $\rho$. Hence, 
  given $x\in\Lambda_N$ and $e=\{i,j\}\in E_N$, it remains to estimate
  \begin{align}
\sum_{\substack{f\in E_N:\\ x\in f}}\!\!\E^N\left[\frac{c_f}{c_e}\right]
      =\!\!\sum_{\substack{f\in E_N:\\ x\in f}}\!\E^{\tilde\Lambda_N}_{W,h}\left[\frac{c_f}{c_e}e^{u_0}\right]
    = \!\!\sum_{k\in(\tilde\Lambda_N\cup\{\rho\})\setminus\{x\}}\!\frac{W_{xk}}{W_{ij}}
       \E^{\tilde\Lambda_N}_{W,h}\left[\frac{e^{u_x+u_k}}{e^{u_i+u_j}}e^{u_0}\right].
       \label{eq:est-sum-weights}
  \end{align}
  We apply first H\"older's inequality and then the bound \eqref{eq:eucl-est-cosh-ui} in
  Theorem~\ref{thm:est-cosh-euclidean} to obtain 
  \begin{align}
    &\E^{\tilde\Lambda_N}_{W,h}[e^{u_x+u_k-u_i-u_j+u_0}]\nonumber\\
    & \le \big(\E^{\tilde\Lambda_N}_{W,h}[e^{5u_x}]\E^{\tilde\Lambda_N}_{W,h}[e^{5u_k}]
    \E^{\tilde\Lambda_N}_{W,h}[e^{-5u_i}]\E^{\tilde\Lambda_N}_{W,h}[e^{-5u_j}]
    \E^{\tilde\Lambda_N}_{W,h}[e^{5u_0}]\big)^{\frac15}\le\cdrei
        \label{eq:bound-2-hoch-6}
  \end{align}
  with the constant $\cdrei$ from there;
  for $k=\rho$, note that $\E^{\tilde\Lambda_N}_{W,h}[e^{5u_\rho}]=1\le\cdrei$.
Note that the theorem is applied to
the translated box $\tilde\Lambda_N$ rather than the original box $\Lambda_N$,
using translational invariance of the weights.
Substituting the bound \eqref{eq:bound-2-hoch-6} in \eqref{eq:est-sum-weights},
we obtain
\begin{align}
  \label{eq:bound-k-over-Wij}
  \sum_{f\in E_N: x\in f}\E^N\left[\frac{c_f}{c_e}\right]
  \le \cdrei\sum_{k\in(\tilde\Lambda_N\cup\{\rho\})\setminus\{x\}}\frac{W_{xk}}{W_{ij}}
  =\frac{K}{W_{ij}},
\end{align}
with $K=\cdrei\sum_{k\in\Z^d\setminus\{x\}}W_{xk}<\infty$; by translation invariance,
the constant $K$ does not depend on $x$.
Having checked all its hypotheses, we apply Theorem~\ref{le:random-conductances-trans}.
The limit
$\rr(W,x\leftrightarrow\infty)=\lim_{N\to\infty}\rr^N(W,x\leftrightarrow\rho)$
does not depend on the choice of the
boxes $\tilde\Lambda_N\uparrow\Z^d$. Finally, since the weights $W_{ij}$ are translation invariant,
it also does not depend on $x$. We conclude that the expected number of visits to $x$ is
uniformly bounded by $K\rr(W,0\leftrightarrow\infty)$. 
\end{proof}

\begin{proof}[Proof of Theorem \ref{thm:vrjp-transient-high-d}]
The proof is parallel to the proof of 
Theorem~\ref{thm:vrjp-transient-long-range}, but with some modifications.
We explain these modifications. This time, we use Model \ref{model:high-d}
and take the shifted box $\tilde\Lambda_N$ again, but keep only edges
$\{i,j\}$ with $W_{ij}>0$. We apply Theorem~\ref{le:random-conductances-trans} in the
same way as before. This time, we verify its three assumptions as follows. 
Assumption 1 holds by hypothesis \eqref{eq:cond2-w}. 
    Assumption 2 follows from
  Rayleigh's monotonicity principle \cite[Section 2.4]{lyons-peres2016}
  and transience of nearest-neighbor
  simple random walk in dimensions $d\ge 3$. 
  Finally, Assumption 3 is verified using Theorem
\ref{th:est-cosh-euclidean-highdim} as follows. We assume $\overline W\ge\weins$ with
$\weins=\max\{\wzwei,10^8\}$ and $\wzwei$ as in Theorem \ref{th:est-cosh-euclidean-highdim}, 
which implies $\frac12\overline W^{1/8}\ge 5=:m$. The steps leading to estimate
\eqref{eq:est-sum-weights} and \eqref{eq:bound-2-hoch-6} remain unchanged, 
with the exception that the constant $\cdrei$ is replaced by $2^{2m+1}=2^{11}$;
cf.\ \eqref{eq:eucl-est-cosh-ui2}.
The steps leading to \eqref{eq:bound-k-over-Wij} remain unchanged, with the exception that
this time the constant $K_x:=2^{11}\sum_{k\in\Z^d\setminus\{x\}}W_{xk}<\infty$
may depend on $x$. However, it is uniformly bounded in $x$ by \eqref{eq:cond2-w}:
$K:=\sup_{x\in\Z_d}K_x<\infty$. 
By the same argument as in the proof of Theorem~\ref{thm:vrjp-transient-long-range},
we conclude that the expected number of visits to $x$ is bounded by $K\rr(W,x\leftrightarrow\infty)$.
If the weights $W_{ij}$ are translation invariant, this bound does not depend on $x$. 
\end{proof}

\subsection{Uniform integrability and monotonicity}
\label{subsec:uniform}

Lemma \ref{le:uniform-integrability} is a consequence of Theorems
  \ref{thm:est-cosh-euclidean} and \ref{th:est-cosh-euclidean-highdim},
  as the following proof shows.

\begin{proof}[Proof of Lemma~\ref{le:uniform-integrability}]
  Note that for $\sigma\in\{\pm 1\}$ and some $m>1$ one has
  $\E^{\Lambda_N}_{W,h}[e^{\sigma mu_i}]\leq\max\{\cdrei,2^{2m+1}\}$
for all boxes $\Lambda_{N}\subset \Z^{d}$ of side length $2^{N}$ and $i\in\Lambda_N$. 
This follows from Theorem~\ref{thm:est-cosh-euclidean} together with translation invariance in the case
of Model~\ref{model:eucl} and from Theorem~\ref{th:est-cosh-euclidean-highdim}  in  case of
Model~\ref{model:high-d}.  Given  an arbitrary finite set $\Lambda'\subset \Z^{d}$ there exists a box $\Lambda_{N}$  of side length $2^{N}$ with $\Lambda'\subseteq \Lambda_{N}$ for $N$ large enough.
We endow both of them with wired boundary conditions. 
Given $i\in \Lambda',$ according to \cite[Theorem~1]{sabot-zeng15}  one can couple $u_{i}^{(\Lambda')}$ and 
 $u_{i}^{(\Lambda_{N})}$ such that $e^{u_{i}^{(\Lambda')}},e^{u_{i}^{(\Lambda_{N})}}$ is a one-step martingale.
 Note that $e^{u}$ is called $\psi$ in \cite{sabot-zeng15}.
Since $x\mapsto x^{\sigma m}$ is convex for $m>1,$ Jensen's inequality gives
$\E^{\Lambda'}_{W,h}[e^{\sigma mu_i}]\leq \E^{\Lambda_N}_{W,h}[e^{\sigma m u_i}]\leq \max\{\cdrei,2^{2m+1}\}$
and the claim follows.  
\end{proof}

We need the following monotonicity result for the environment in the weights. 

\begin{fact}[Poudevigne's monotonicity theorem {\cite[Theorem~6]{poudevigne22}}]
  \label{fact:poudevigne}
  \mbox{}\\
  Consider two families of non-negative edge weights and pinnings $W^+,h^+$ and $W^-,h^-$
  on the same vertex set $\Lambda$, satisfying
  $W^+\ge W^-$ and $h^+\ge h^-$ componentwise. 
  For any convex function $f:[0,\infty)\to\R$ and all $i\in\Lambda$, one has 
\begin{align}
  \E^\Lambda_{W^+,h^+}\left[f(e^{u_i})\right]\le
  \E^\Lambda_{W^-,h^-}\left[f(e^{u_i})\right]. 
\end{align}
In particular, for $m\ge 1$ and $\sigma\in\{\pm 1\}$, 
\begin{align}
  \label{eq:inequ-cosh}
  \E^\Lambda_{W^+,h^+}\left[e^{\sigma m u_i}\right]\le
  \E^\Lambda_{W^-,h^-}\left[e^{\sigma m u_i}\right]. 
\end{align}
\end{fact}
Recall that $\rho$ denotes the wiring point. In Poudevigne's notation, 
the joint law of
$G^\pm(\rho,i)/G^\pm(\rho,\rho)$, $i\in\Lambda$, from \cite[Theorem~6]{poudevigne22} coincides with
the joint law of $e^{u_i}$, $i\in\Lambda$, with respect to $\nu^\Lambda_{W^\pm,h^\pm}$. 
Taking the convex function $f(x)=e^{\sigma m u}$, $m\ge 1$, yields
\eqref{eq:inequ-cosh}.
Using this monotonicity, we derive bounds for the environment
in the coming sections.

\section{Bounding environments for hierarchical and long-range weights}

\subsection{Hierarchical weights}\label{se:boundslongrange}

To prove Theorem~\ref{thm:est-cosh-euclidean} we compare a box in the Euclidean lattice with long range interactions with
the corresponding box in a hierarchical lattice as follows.

\begin{model}[Complete graph with hierarchical interactions]\label{model:hierarchical}
Fix $N\in\N$. We consider the vertex set consisting of the set of
leaves $\Lambda_N=\{0,1\}^N$ of the binary tree $\T=\cup_{n=0}^N\{0,1\}^n$.
For $i=(i_0,\ldots,i_{N-1}),j=(j_0,\ldots,j_{N-1})\in\{0,1\}^N$, the hierarchical
distance is defined by 
$d_H(i,j)=\min(\{l\in\{0,\ldots,N-1\}:\, i_k=j_k\text{ for all }k\ge l\}\cup\{N\})$;
it is the distance
to the least common ancestor in $\T$. We endow $\T$ with
hierarchical weights
\begin{align}
  W_{ij}^H=w^H(d_H(i,j)), \quad i,j\in\Lambda_N, i\neq j, 
\end{align}
with a weight function $w^H:\N\to(0,\infty)$ and uniform pinning $W_{i\rho}=h^H>0$
for all $i\in\Lambda_N$. 
\end{model}
An illustration of $\T$ is given in Figure \ref{fig:hbaum}. 
 \begin{figure}
 \centerline{ \includegraphics[width=7cm]{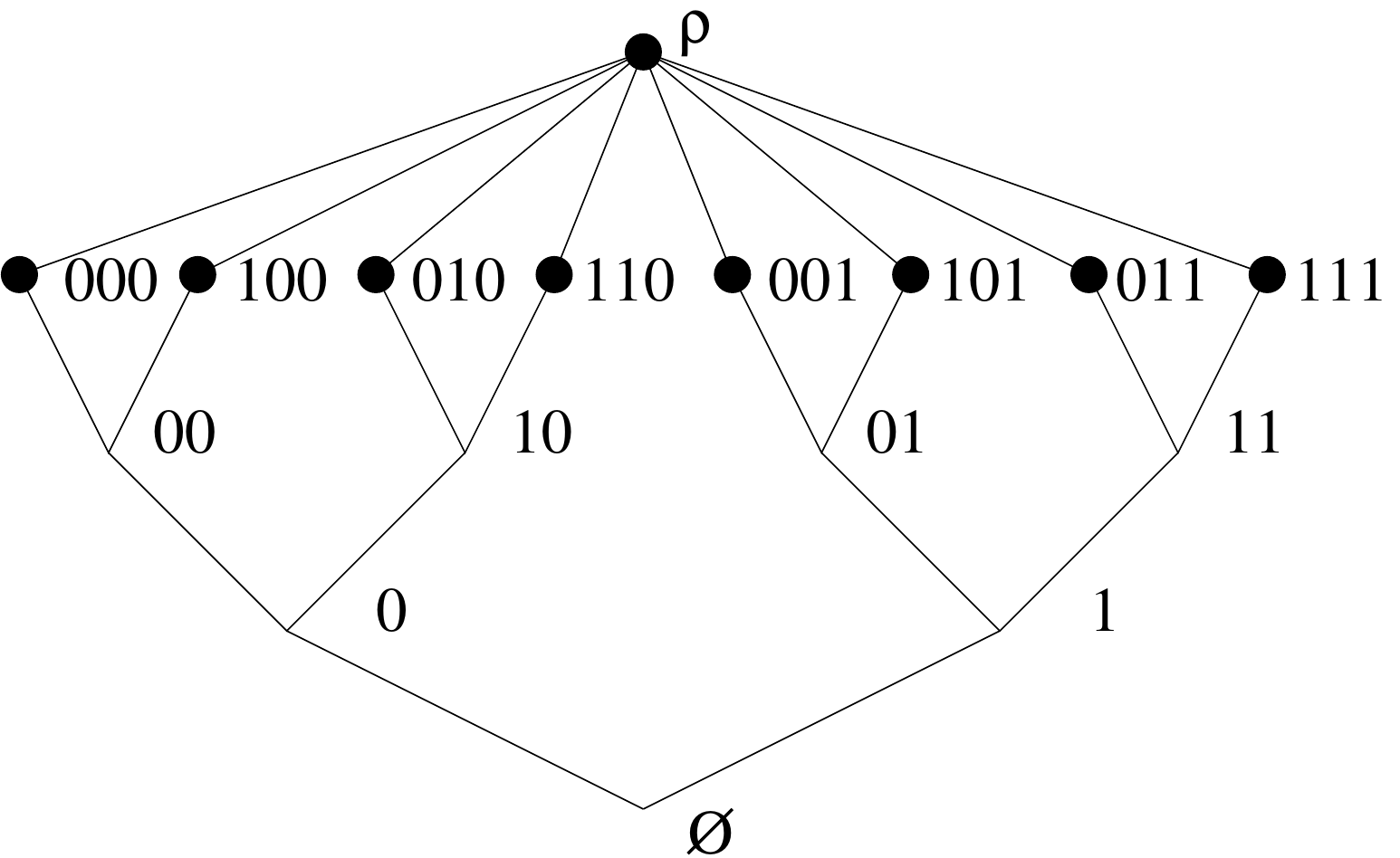}}
 \caption{hierarchical graph for $N=3$ together with the pinning point. In
     Model~\ref{model:hierarchical} any two leaves of the binary tree have
   a positive interaction.}
  \label{fig:hbaum}
\end{figure}

\begin{theorem}[Bounds on the environment -- hierarchical case]
  \label{thm:hierarchical}
  Let $\overline W^H>0$ and $\alpha>1$. 
    Consider Model \ref{model:hierarchical}
  and assume that the weight function
  $w^H$ and the pinning $h^H$ satisfy
  \begin{align}
    & w^H(l)\ge \overline W^H 2^{-2l}l^\alpha \text{ for } l\in\N, \quad
    h^H\ge \overline W^H 2^{-(2+N)} (N+1)^\alpha . 
     \label{eq:ass-weights-pinning}
\end{align}
Then, for all $i\in\Lambda_N$, $\sigma\in\{\pm 1\}$, and all $m\ge 1$, one has  
\begin{align}
  \label{eq:hierarchical-est-cosh-ui}
  \E_{W^H,h^H}^{\Lambda_N}[e^{\sigma m u_i}]\le \ceins 
\end{align}
with a constant $\ceins=\ceins(\overline W^H,\alpha,m)$, uniformly in $N$.
An explicit form of $\ceins$ is given in \eqref{eq:bound-1-d}, below.
\end{theorem}

The proof of this theorem is based on the following general lemma,
which holds on any finite graph. The underlying strategy has
been suggested by an anonymous associate editor. 

\begin{lemma}
  \label{le:bound-log-cosh-u-i}
Let $G=(\Lambda,E_\Lambda)$ be a finite undirected connected
  graph endowed with positive weights $W=(W_e)_{e\in E_\Lambda}$
  on the edges. Fix a pinning point $\rho\in\Lambda$. 
We write $\E_W^{\Lambda,\rho}$ for the expectation with respect to the
probability measure $\nu^{\Lambda\setminus\{\rho\}}_{\tilde W,h}$, cf.\ \eqref{eq:u-marginal}, 
with weights $\tilde W_{ij}=W_{ij}$ for edges $\{i,j\}\subseteq\Lambda$
and pinnings $h_i=W_{i\rho}$ if $\{i,\rho\}$ is an edge and $h_i=0$
otherwise.

Let $\pi=(i=\pi_1,\ldots,\pi_{N+1}=\rho)$ be a self-avoiding
path from some vertex $i\in\Lambda$ to $\rho$.
 Then, for all $m\ge 1$ and $\sigma\in\{\pm 1\}$, one has 
  \begin{align}
    \E_W^{\Lambda,\rho}[e^{\sigma m u_i}]
    \le \exp\left(\frac{( 2m+1)^2}{8}\sum_{k=1}^N W^{-1}_{\pi_{k}\pi_{k+1}}\right). 
  \end{align}
\end{lemma}
\begin{proof}
  Let $T\subseteq E_\Lambda$ be a spanning tree of $G$ such that
  all edges $\{\pi_k,\pi_{k+1}\}$ of the path $\pi$ are contained
  in $T$. Define $W^-=(W^-_e)_{e\in E_\Lambda}$ by
  $W^-_e:=1_{\{ e\in T\}}W_e$. By Fact \ref{fact:poudevigne}, one has 
  \begin{align}
    \label{eq:cosh-hoch-m-bound}
    \E_W^{\Lambda,\rho}[e^{\sigma m u_i}]\le \E_{W^-}^{\Lambda,\rho}[e^{\sigma m u_i}]. 
  \end{align}
  Because $\pi$ is a path in the tree $T$, the increments $u_{\pi_k}-u_{\pi_{k+1}}$,
$1\le k\le N$, along the path 
are independent with respect to  $\nu_{W^-}^{\Lambda,\rho}$ and
distributed according to
\begin{align}
  \nu_{W^-}^{\Lambda,\rho}(u_{\pi_k}-u_{\pi_{k+1}}\in A)
  =\sqrt{\frac{W_{\pi_k\pi_{k+1}}}{2\pi}}\int_A
  e^{-\frac12 u}e^{-{W_{\pi_k\pi_{k+1}}[\cosh u-1]}}\, du
\end{align}
for any Borel set $A\subseteq\R$. 
Consequently, for $\sigma\in\{\pm 1\}$, using $u_i=u_{\pi_1}$ and $u_{\pi_{N+1}}=u_\rho=0$, we obtain 
\begin{align}
\label{eq:exp-sigma-u-i-as-prod}
  \E_{W^-}^{\Lambda,\rho}[e^{\sigma mu_i}]
  =\prod_{k=1}^N \E_{W^-}^{\Lambda,\rho}[e^{\sigma m(u_{\pi_k}-u_{\pi_{k+1}})}]. 
\end{align}
For any $k\in\{1,\ldots,N\}$, abbreviating
$w:=W_{\pi_k\pi_{k+1}}$, we estimate
\begin{align}
  &\E_{W^-}^{\Lambda,\rho}[e^{\sigma m(u_{\pi_k}-u_{\pi_{k+1}})}]
  =\sqrt{\frac{w}{2\pi}}\int_\R
  e^{\sigma mu}e^{-\frac12 u}e^{-{w[\cosh u-1]}}\, du\nonumber\\
  \le& \sqrt{\frac{w}{2\pi}}\int_\R
  e^{\sigma mu}e^{-\frac12 u}e^{-\frac12{wu^2}}\, du
       =\exp\frac{(2\sigma m-1)^2}{8w}
       \le\exp\frac{( 2m+1)^2}{8w}. 
\end{align}
Inserting this in \eqref{eq:exp-sigma-u-i-as-prod} and the result in
\eqref{eq:cosh-hoch-m-bound}, the claim follows. 
\end{proof}

Next, we apply this lemma to deduce the bounds for the hierarchical model,
  using an antichain in the tree $\T$.

\begin{proof}[Proof of Theorem~\ref{thm:hierarchical}]
  By uniform pinning and the symmetry of the weights induced by the hierarchical structure,
  the distribution of $u_i$ is the same
  for all $i\in\Lambda_N$. Therefore, it is enough to prove \eqref{eq:hierarchical-est-cosh-ui}
  for the vertex $i_1:=0^N:=(0,\ldots,0)$. Corollary 2.3 of \cite{disertori-merkl-rolles2020}
  tells us that the distribution of $e^{u_{i_1}}$ in this model
  coincides with the distribution of $e^{u_{i_1}}$ in the corresponding effective model on
  the complete graph with vertex set $A=\{j,i_1,\ldots,i_N\}\subseteq\T$, where
  $j:=(1,0^{N-1})=(1,0,\ldots,0)\in\{0,1\}^N$ and 
  $i_l:=(1,0^{N-l})\in\{0,1\}^{N-l+1}$ for $2\le l\le N$. This is a maximal antichain in the language of
  Section 2.3 of \cite{disertori-merkl-rolles2020} (cf.\ Figure \ref{fig:antichain}).
  Let us review this effective model. 
  \begin{figure}
 \centerline{ \includegraphics[width=8.8cm]{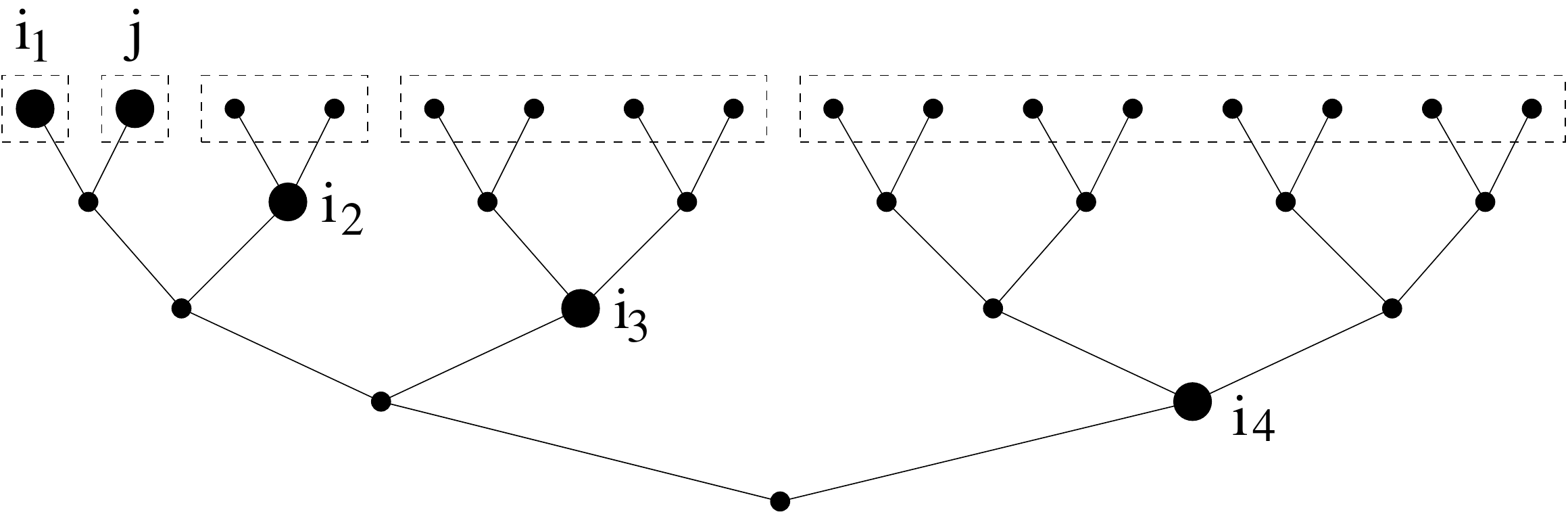}}
  \caption{the maximal antichain $A$}
  \label{fig:antichain}
\end{figure}
Weights and pinning on the antichain, viewed as a complete graph, are given in 
\cite[(2.13) and (2.14)]{disertori-merkl-rolles2020}. However, only the 
  values of the following nearest-neighbor weights and pinnings are relevant for the
  subsequent argument, while the values of all other weights do not play any role: 
  $W_{i_{l-1},i_l}=2^{(l-2)+(l-1)}w^H(l)=2^{2l-3}w^H(l)$, $h_{i_l}:=2^{l-1}h^H$, $2\le l\le N$, $h_{i_1}:=h^H$.
  Here, $2^{l-1}$ equals the number of leaves in $\T$ above the vertex $i_l=(1,0^{N-l})$.
  Corollary 2.3 of \cite{disertori-merkl-rolles2020} yields 
  \begin{align}\label{eq:hierantichain}
    \E_{W^H,h^H}^{\Lambda_N}[e^{m\sigma u_{i_1}}]
    =\E_{W,h}^A[e^{m\sigma u_{i_1}}].
  \end{align}
  We regard the pinning point $\rho$ as an additional vertex at the end of the antichain
  $j,i_1,\ldots,i_N$, writing $i_{N+1}:=\rho$. By assumption
    \eqref{eq:ass-weights-pinning}, one has
    \begin{align}
      W_{i_{l-1}i_l}\ge\tfrac18\overline W^Hl^\alpha 
    \end{align}
    for $l\in\{2,\ldots,N+1\}$. Indeed, in the case $l\le N$,
    this follows from $W_{i_{l-1}i_l}=2^{2l-3}w^H(l)\ge\frac18\overline W^H l^\alpha$,
    and in the case $l=N+1$ from
    \begin{align}
      W_{i_N,i_{N+1}}=h_{i_N}=2^{N-1}h^H\ge\tfrac18\overline W^H (N+1)^\alpha. 
    \end{align}
    We remark that the assumption on $h^H$ in \eqref{eq:ass-weights-pinning}
      is chosen such that the pinning case $l=N+1$ gives a lower bound of the
      same form as the other cases $l\le N$.
    Hence, we can apply Lemma \ref{le:bound-log-cosh-u-i} to
    the complete graph with vertex set $\{j,i_1,\ldots,i_{N+1}\}$ and the
    path $\pi=(i_1,\ldots,i_{N+1})$ to obtain
    \begin{align} 
    \label{eq:bound-1-d}
    \E_{W,h}^A[e^{m\sigma u_{i_1}}]
      \le \exp\left(\frac{(2m+1)^2}{\overline W^H}\sum_{l=2}^\infty l^{-\alpha}\right)
      =:\ceins=\ceins(\overline W^H,\alpha,m). 
  \end{align}
  \end{proof}

\subsection{Long-range weights}
\label{se:long-range-weights}

Fix an arbitrary dimension $d\ge 1$ and $N\in\N$.
Take the box $\Lambda_N:=\{0,1,\ldots,2^N-1\}^d$.
 \begin{figure}
 \centerline{ \includegraphics[width=5cm]{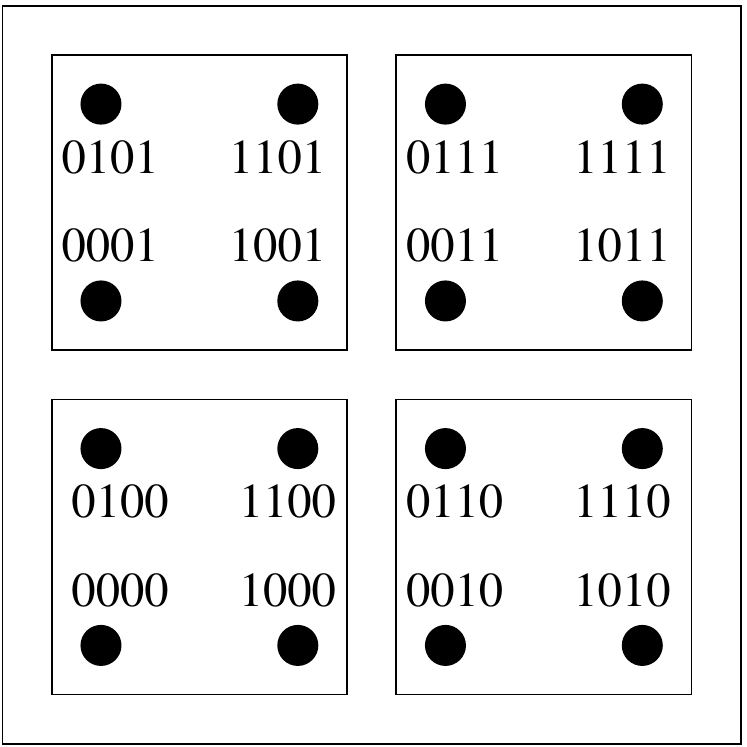}}
  \caption{illustration of the map $\phi$ for $N=2$ and $d=2$}
  \label{fig:hbox}
\end{figure}
In order to compare $\Lambda_N$ with $\{0,1\}^{Nd}$, we define (cf.\ Fig.\ref{fig:hbox})
\begin{align}
  \label{eq:def-phi}
  \phi:\Lambda_N\to\{0,1\}^{Nd},\quad 
  \phi(i_0,\ldots,i_{d-1})=(z_0,\ldots,z_{Nd-1}),
\end{align}
where $z_n\in\{0,1\}$ denotes the $\lfloor n/d\rfloor$-th digit in the
binary expansion $i_l=\sum_{k=0}^{N-1}z_{dk+l}2^k$, with
$l=n\operatorname{mod}d$. The map $\phi$ is a bijection.
The elements of $\{0,1\}^{Nd}$ can be visualized as the leaves of a binary tree.
On this set, the hierarchical distance, i.e.\ the distance to the
least common ancestor in the binary tree, is given by
\begin{align}
  d_H(z,z')=\min(\{n\in\{0,\ldots,Nd-1\}:\, z_m=z_m'\text{ for all }m\ge n\}\cup\{Nd\}). 
\end{align}
On $\Lambda_N$, the natural definition of hierarchical distance would be
  $\lceil d_H(\phi(i),\phi(j))/d\rceil$. With this choice, in Figure \ref{fig:hbox}, 
  any two points in the same small box would
  have distance 1 from each other, while
  points in different boxes would have distance 2.
In the following lemma, we compare
the maximum norm distance with the hierarchical distance.

\begin{lemma}[Comparison of $d_H$ and $\|\cdot\|_\infty$]
\label{le:comparison-dH-eucl}
  For all $i,j\in\Lambda_N$, one has
  \begin{align}
    2\cdot 2^{d_H(\phi(i),\phi(j))/d}\ge 2^{\lceil d_H(\phi(i),\phi(j))/d\rceil}
    >\|i-j\|_\infty. 
  \end{align}
\end{lemma}
\begin{proof}
  For $i=j$, the claim is true. We take $i\neq j$ in $\Lambda_N$ and set
  $z=\phi(i)$, $z'=\phi(j)$, and $r=d_H(z,z')$. One has 
  \begin{align}
    \|i-j\|_\infty = \max_{0\le l\le d-1}|i_l-j_l|
    = \max_{0\le l\le d-1}\left|\sum_{k=0}^{N-1}(z_{dk+l}-z_{dk+l}')2^k\right|.
  \end{align} 
  The only non-zero contributions in the above sum come at most from $k$ satisfying
  $dk+l\le r-1$ which implies $k\le (r-1)/d$. It follows that
\begin{align}
  \|i-j\|_\infty
  &\le \max_{0\le l\le d-1}\sum_{k=0}^{\lfloor (r-1)/d\rfloor}|z_{dk+l}-z_{dk+l}'|2^k\nonumber\\
  &\le \sum_{k=0}^{\lfloor (r-1)/d\rfloor}2^k<2^{\lfloor (r-1)/d\rfloor+1}
  =2^{\lceil r/d\rceil},
\end{align}
where we use $|z_{dk+l}-z_{dk+l}'|\le 1$ and $r\in\N_0$.
\end{proof}

Next, we show how the bound on the environment for the Euclidean box with long-range
interactions follows from Theorem~\ref{thm:hierarchical}. 

\begin{proof}[Proof of Theorem~\ref{thm:est-cosh-euclidean}]
  For $m\ge 1$, we apply Fact \ref{fact:poudevigne} to compare the Euclidean long-range
  model, i.e.\ Model~\ref{model:eucl}, but restricted to the box $\Lambda_N$
    with wired boundary conditions, with a hierarchical model defined on the same  box.
Remember that in the Euclidean model, we take 
long-range weights $W^+_{ij}:=W_{ij}=w(\|i-j\|_\infty)$ as in Model~\ref{model:eucl}
and pinning $h^+_j:=h_j$ given
in \eqref{eq:pinning-wired-bc}. 
In the hierarchical model, we take weights 
\begin{align}
  W^-_{ij}:=W^H_{ij}=w^H(d_H(\phi(i),\phi(j)))\text{ with }
  w^H(l):=w(2^{\lceil l/d\rceil}) 
\end{align}
and uniform pinning given by $h^-:=h^H:=\min_{j\in\Lambda_N}h_j$. 
By Lemma~\ref{le:comparison-dH-eucl} and monotonicity of $w$,
\begin{align}
  W_{ij}=w(\|i-j\|_\infty)\ge w\left(2^{\lceil d_H(\phi(i),\phi(j))/d\rceil}\right)
  = W^H_{ij},\quad h_i\ge h^H,
\end{align}
for $i,j\in\Lambda_N$. By monotonicity \eqref{eq:inequ-cosh}, for all $i\in\Lambda_N$,
$\sigma\in\{\pm 1\}$, and $m\ge 1$, one has 
\begin{align}
  \E^{\Lambda_N}_{W,h}\left[e^{\sigma m u_i}\right]\le
  \E^{\Lambda_N}_{W^H,h^H}\left[e^{\sigma m u_i}\right]. 
\end{align}
We identify $\Lambda_N$ with $\{0,1\}^{Nd}$ via the bijection $\phi:\Lambda_N\to\{0,1\}^{Nd}$ from
\eqref{eq:def-phi}. As a result we obtain Model~\ref{model:hierarchical} with $N$
replaced by $Nd$. 
We apply Theorem~\ref{thm:hierarchical} with $\overline W^H=\overline Wd^{-\alpha}2^{-2d}$
and 
\begin{align}
   w^H(l)=&w(2^{\lceil l/d\rceil})
            \ge \overline W \frac{( \lceil l/d\rceil)^\alpha}{2^{2d\lceil l/d\rceil}}
            \ge \overline W \frac{(l/d)^\alpha}{2^{2d(1+l/d)}}
            =\overline W^H 2^{-2l}l^\alpha, 
\end{align}
where we used \eqref{eq:cond1-w}.
We estimate the pinning using that $w$
is decreasing 
\begin{align}
  h^H=\min_{i\in\Lambda_N}h_i
         =\min_{i\in\Lambda_N}\sum_{j\in\Z^d\setminus\Lambda_N}\hspace{-0.5em}w(\|i-j\|_\infty)
       \ge\hspace{-1.5em}\sum_{\substack{j\in\Z^d:\\2^N\le \|j\|_\infty< 2^{N+1}}}
  \hspace{-2em}w(\|j\|_\infty)\ge w(2^{N+1})
       \hspace{-1.5em}\sum_{\substack{j\in\Z^d:\\ 2^N\le\|j\|_\infty<2^{N+1}}}\hspace{-1.5em}1.
\end{align}
The last sum has at least $2^{(N+1)d}$ summands, which is a lower bound on the cardinality of
the set $\{z\in\Z:\,2^N\le|z|<2^{N+1}\}\times\{j\in\Z^{d-1}:\,\|j\|_\infty<2^{N+1}\}$. Hence,
we obtain 
\begin{align}
    &h^H\ge
      2^{(N+1)d} w(2^{N+1})\ge 
      2^{(N+1)d} \overline W \frac{(N+1)^\alpha}{2^{2(N+1)d}}
      \ge \overline W d^{-\alpha}2^{-(2+2d+Nd)} (Nd+1)^\alpha\nonumber\\
  &= \overline W^H 2^{-(2+Nd)} (Nd+1)^\alpha. 
\end{align}
Note that we drop a factor of $2^{2+d}$ in the last inequality,
  uniformly in $N$. However, keeping it would not improve the result.
Thus, assumption \eqref{eq:ass-weights-pinning} of Theorem~\ref{thm:hierarchical}
is satisfied with $N$ replaced by $Nd$ and the estimate \eqref{eq:eucl-est-cosh-ui}
follows with the constant
\begin{align}
  \label{eq:def-cdrei}
  \cdrei(\overline W,d,\alpha,m):=\ceins(\overline W d^{-\alpha}2^{-2d},\alpha,m)
  =\exp\left(\frac{d^\alpha 2^{2d}(2m+1)^2}{\overline W}\sum_{l=2}^\infty l^{-\alpha}\right)
  >1.
\end{align}
\end{proof}

\paragraph*{Discussion}
Note that Theorem \ref{thm:hierarchical} applies in the special case 
$W_{ij}\propto 2^{-(d+2)\lceil d_H(\phi(i),\phi(j))/d\rceil}$ in hierarchical dimension $d\ge 3$, which corresponds to
$w^H(l)\propto 2^{-(d+2)\lceil l/d\rceil}$. In this case, the Green's function
for the Markov jump process on the corresponding infinite hierarchical lattice
with jump rates $W_{ij}$ decays as $G(i,j)=\text{const}\cdot 2^{-(d-2)\lceil d_H(\phi(i),\phi(j))/d\rceil}$, $i\neq j$, see 
\cite[formula (1.3)]{brydges-evans-imbrie1992}; note that our $\lceil d_H(\phi(i),\phi(j))/d\rceil$ 
corresponds to $\log_2|i-j|_H$ in the case $L=2$ in this reference, i.e.\
$G(i,j)=\text{const}/|i-j|_H^{d-2}$. This case is interesting since it mimics the decay of the
inverse Laplacian $-\Delta^{-1}$ on the Euclidean lattice in dimensions
$d\ge 3$.

\section{Bounding environments in high dimension}
\label{sec:H22}

The proof of Theorem~\ref{th:est-cosh-euclidean-highdim}, which is the
goal of this section, is based on the description of the 
random environment for VRJP in terms of the non-linear supersymmetric hyperbolic sigma
model, $\htwo$ model for short. The key tools are a bound for
the expectation of $(\cosh(u_i-u_j))^m$, which is proved in a slightly more
general form in formula \eqref{eq:bound-prod-Be-without-protection} below,
and an estimate from \cite{disertori-spencer-zirnbauer2010}.

\subsection{The non-linear supersymmetric hyperbolic sigma model}
\label{subsec:H22}
The supersymmetric hyperbolic sigma $\htwo$  model  was introduced by
Zirnbauer   \cite{zirnbauer-91} in the context of quantum diffusion.
It can be seen as a statistical mechanical spin model where the spins take 
 values in the supermanifold $H^{2|2}.$ This is the hyperboloid $H^{2}=
 \{(x,y,z)\in \R^{3}:\ x^{2}+y^{2}-z^{2}=-1,\ z>0 \}$ extended by two Grassmann,
 i.e.\ anticommuting, variables $\xi,\eta$ with the constraint
 $z^{2}=1+x^{2}+y^{2}+2\xi \eta.$ For more mathematical background 
 see  \cite{phdthesis-swan2020},
 \cite[Appendix]{disertori-merkl-rolles2020} and 
 \cite{disertori-spencer-zirnbauer2010}.
Changing to horospherical coordinates $(x,y,\xi,\eta )\mapsto (u,s,\overline \psi,\psi),$
cf.\ \cite[eq.\ (2.7)]{disertori-spencer-zirnbauer2010},
where $u,s$ are real and $\overline \psi,\psi$ are Grassmann variables,
the model can be defined as follows.

We consider a finite undirected graph with vertex set $\Lambda$ and edge set $E_\Lambda$. 
Every edge $e=\{i,j\}\in E_\Lambda$ is given a positive weight $W_{ij}=W_{ji}=W_e>0$.
To every vertex $i\in\Lambda$, we associate a pinning strength $h_i\ge 0$
in such a way that every connected component of the graph contains at least one
vertex $i$ with $h_i>0$. 
It is convenient to treat pinning 
and interaction in a unified way. Therefore, we extend the vertex set $\Lambda$
by a special vertex
$\rho\notin\Lambda$ called ``pinning point'' or ``wiring point''.
In addition, we introduce pinning edges
connecting the pinning point with the other vertices. Since we will study the
dependence of the field at any pair of points, we will also add edges connecting
vertices without direct interaction. Therefore, we consider the undirected
\emph{complete} graph $G=(\Lambda\cup\{\rho\},E)$ without direct loops.
Every edge $e\in E$ is given a weight $W_e\ge 0$ as follows.
For $e\in E_\Lambda$, $W_e>0$ is the weight defined above. The weight of pinning edges
$e=\{i,\rho\}$, $i\in\Lambda$, is defined to be the pinning strength $W_e=h_i$.
All other edges $e$ are given weight $W_e=0$.
The subgraph consisting of edges with positive weights is denoted by 
\begin{align}
  \label{eq:def-G-plus}
G_+=(\Lambda\cup\{\rho\},E_+),\quad\text{where }  E_+:=\{ e\in E: W_e>0\}.
\end{align}
Note that $E_+$ consists of $E_\Lambda$ and the pinning edges with positive weight, 
and our assumption on the pinning is equivalent to require that the graph
$G_+$ is connected. 

We assign to every vertex $i$
two real-valued variables $u_i$ and $s_i$ and two
Grassmann variables $\overline\psi_i$ and $\psi_i$. We set
$u_\rho=s_\rho:=0=:\overline\psi_\rho=\psi_\rho$. 
Following  \cite[eq.\ (2.9-2.10)]{disertori-spencer-zirnbauer2010},
for all $e=\{i,j\}$ with $i,j\in\Lambda\cup\{\rho\}$ with $i\neq j$, we set
\begin{align}
  \label{eq:def-Bij}
  B_{ij}=&B_e=B_e(u,s):=\cosh(u_i-u_j)+\frac12(s_i-s_j)^2e^{u_i+u_j},\\
  \label{eq:def-Sij}
  S_{ij}=&S_e=S_e(u,s,\overline\psi,\psi)
           :=B_e+(\overline\psi_i-\overline\psi_j)(\psi_i-\psi_j)e^{u_i+u_j}. 
\end{align} 
For each vertex $i\in\Lambda$, let 
\begin{align}
 \label{eq:def-Bi}
  B_i:=B_{i\rho}=\cosh u_i+\frac12s_i^2e^{u_i},\quad
  S_i:=S_{i\rho}=B_i+\overline\psi_i\psi_ie^{u_i}.   
\end{align}
The action of the non-linear hyperbolic
sigma model in horospherical coordinates is given by, cf.\ \cite[eq.\ (2.8) with $z_{i} =S_i$]{disertori-spencer-zirnbauer2010},
\begin{align}\label{eq:defA}
  A=A(u,s,\overline\psi,\psi):=\sum_{e\in E_\Lambda}W_e(S_e-1)
  +\sum_{i\in\Lambda}h_i(S_i-1)
  =\sum_{e\in E_+}W_e(S_e-1).
\end{align}
The superintegration form of the model is given by, \cite[eq.\ (2.12-2.13)]{disertori-spencer-zirnbauer2010},
\begin{align}
  \sk{f} = \sk{f}^{\Lambda }_{W,h}:=\int_{\R^\Lambda\times\R^\Lambda}
  \prod_{i\in \Lambda}\frac{e^{-u_i}}{2\pi}\, du_i\, ds_i\, 
  \partial_{\overline\psi_i}\partial_{\psi_i}\,
  f e^{-A}
\end{align}
for any function of the variables $u_i,s_i,\overline\psi_i,\psi_i$
such that the integral exists. 
For any integrable function $f=f(u,s)$ which does not depend on the
Grassmann variables, one has
\begin{align}
  \label{eq:def-expectation}
  \sk{f}=\E[f]= \E^{\Lambda }_{W,h}[f]:=&\int_{\R^\Lambda\times\R^\Lambda}f(u,s)\, \mu^{\Lambda }_{W,h}(du\, ds), 
\end{align}
where $\mu^{\Lambda }_{W,h}$ is the probability measure on $\R^\Lambda\times\R^\Lambda$ given by
\begin{align}
  \mu(du\, ds)=  \mu^{\Lambda }_{W,h}(du\, ds):= &
  e^{-\sum_{e\in E_\Lambda}W_e(B_e-1)-\sum_{i\in\Lambda}h_i(B_i-1)}\det D
  \prod_{i\in \Lambda}\frac{e^{-u_i}}{2\pi}\, du_i\, ds_i\nonumber\\
=&
  e^{-\sum_{e\in E_+}W_e(B_e-1)}\det D
  \prod_{i\in \Lambda}\frac{e^{-u_i}}{2\pi}\, du_i\, ds_i,
   \label{eq:def-mu}
\end{align}
and $D=D(u)=(D_{ij})_{i,j\in \Lambda}$ is the weighted Laplacian matrix on $G_+$
with entries 
\begin{align}
  \label{eq:representation-D-with-pinning-old}
  D_{ij}=\left\{
  \begin{array}{ll}
    -W_{ij}e^{u_i+u_j} & \text{if }\{i,j\}\in E_+,\\
    \sum_{k\in \Lambda\cup\{\rho\}:\{i,k\}\in E_+}W_{ik}e^{u_i+u_k} & \text{if }i=j,\\
    0 & \text{otherwise.}
  \end{array}
  \right.
\end{align}
Note that for $i,j\in\Lambda$ the condition $\{i,j\}\in E_+$ is equivalent
to $\{i,j\}\in E_\Lambda$, and that the diagonal entries can also be written
in the form $D_{ii}=\sum_{k\in \Lambda:\{i,k\}\in E_\Lambda}W_{ik}e^{u_i+u_k}+h_ie^{u_i}$. 
Our connectedness assumption on the graph $G_+$ 
guarantees that the matrix $D$ is positive definite. In particular, it is invertible.
The fact that $\mu^{\Lambda }_{W,h}$ is a probability measure
was shown in \cite{disertori-spencer-zirnbauer2010} using supersymmetry. 
Since the $s$-variables, conditionally on $u$, are normally distributed
with covariance $D^{-1}$, and by the matrix-tree theorem
$\det D= \sum_{T\in \mathcal{S}}\prod_{\{i,j \}\in T} W_{ij}e^{u_{i}+u_{j}},$
the $u$-marginal of $\mu^{\Lambda }_{W,h}$ is
given by \eqref{eq:u-marginal}. In particular, for functions depending only
on $u$, the expectations with respect to $\mu^{\Lambda }_{W,h}$ and $\nu^{\Lambda }_{W,h}$
coincide, which allows us to use the same notation $\E^{\Lambda }_{W,h}$.

\subsection{A Ward identity and some applications}
\label{se:ward-identity}

In this section, $G=(\Lambda\cup\{\rho\},E)$ is the complete graph with
weights $W_e\ge 0$, $e\in E$, such that $G_+$ is connected, cf.\
  formula~\eqref{eq:def-G-plus}. 
We use the $\htwo$-model with Grassmann variables still present. 
Note that real functions of even elements of the Grassmann algebra are defined as
formal Taylor series in the Grassmann variables, which are nilpotent.
For example, for an edge $e=\{i,j\}$, the expression $f(S_e)$ is defined by 
$f(S_e)=f(B_e+n_e):=f(B_e)+f'(B_e)n_e$ with
$n_e=n_{ij}=(\overline\psi_i-\overline\psi_j)(\psi_i-\psi_j)e^{u_i+u_j}$.
This Taylor expansion stops after the first order term because 
$n_e^2=0$ by the anticommutativity of the
Grassmann variables.

We will use the following localization result from
\cite{disertori-spencer-zirnbauer2010}.
\begin{lemma}[Ward identity]
  \label{le:ward-id}
  Let $f_e:[1,\infty)\to\R$, $e\in E$,
  be a family of continuously differentiable functions
  such that each $f_e$ and its derivative $f_e'$ are bounded
  by polynomials.
  Then it holds 
\begin{align}
  \label{eq:general-ward-id}
\sk{\prod_{e\in E}f_e(S_e)}=\prod_{e\in E}f_e(1). 
\end{align}  
\end{lemma}
\begin{proof}
  This is a special case of the localization result in
  \cite[Prop.\ 2 in the appendix]{disertori-spencer-zirnbauer2010}.
  The key points are that any $S_e$ is annihilated by the supersymmetry
  operator used in the proposition and $S_{ij}=1$ holds whenever
  $u_k=s_k=\overline{\psi}_k=\psi_k=0$ for $k\in\{i,j\}$. 
\end{proof}

\paragraph*{Notation}
The following notation will be technically convenient. 
The two vertices incident to an edge $e\in E$ are denoted by $e_+$ and $e_-$. This
gives every edge a direction $e_+\to e_-$, which has no physical meaning, but is
used for bookkeeping purposes. Let
\begin{align}
F\in\{-1,0,1\}^{\Lambda\times E}, \quad F_{k,e}=1_{\{ e_+=k\}}-1_{\{ e_-=k\}} 
  \quad\text{for } k\in \Lambda,\, e\in E
\end{align}
denote the signed incidence matrix of $ G$ with $\rho$ removed,
but keeping the pinning edges.
We consider the diagonal matrix 
\begin{align}
  \W=\W(u):=\diag(\W_e,e\in E)\in\R^{ E\times E}
  \quad \text{with}\quad\W_e:=W_ee^{u_{e_+}+u_{e_-}}.
\end{align}
With these notations, we reformulate the weighted Laplacian matrix $D$ from
\eqref{eq:representation-D-with-pinning-old} as 
\begin{align}
  D=F\W F^t\in\R^{\Lambda\times \Lambda}. 
\end{align}
We will also need the following matrices: 
\begin{align}
  &Q=Q^G(u,s):=\diag(Q_e,e\in E) \quad\text{with} \quad Q_e:=\frac{e^{u_{e_+}+u_{e_-}}}{B_e(u,s)},
  \\
  &\G=\G^G_W(u,s):=\sqrt{Q}F^tD^{-1}F\sqrt{Q}\in\R^{ E\times E}.
    \label{eq:def-G}
\end{align}
Let $m_e\ge 0$, $e\in E$.
We set
\begin{align}
  \label{eq:def-M-G}
  M:=\diag(m_e,e\in E).
\end{align}

\begin{lemma}
  \label{le:prod-equals-1}
  The following hold:
  \begin{align}
    \label{eq:susy-with-det}
    &\E\left[\prod_{e\in E}B_e^{m_e}\cdot \det(\Id-M\mathcal{G})\right]=1.
  \end{align}  
  If in addition 
  $0\le m_e<W_e$ for $e\in E_+$ defined in \eqref{eq:def-G-plus}
  and $m_e=0$ otherwise, the equality
  implies
  \begin{equation}\label{eq:bound-prod-Be-without-protection}
 \E\left[\prod_{e\in E_+} (\cosh (u_{e_{+}}-u_{e_{-}}))^{m_e}\right]\leq    \E\left[\prod_{e\in E_+}B_e^{m_e}\right]\le
    \prod_{e\in E_+}\frac{1}{1-\frac{m_e}{W_e}}
      .
  \end{equation}
\end{lemma} 
\begin{proof}
   Ward identity \eqref{eq:general-ward-id}
   with $f_e(x)=x^{m_e}$  reads
  \begin{align}
  \label{eq:ward-id-without-protection}
    \sk{\prod_{e\in E}S_e^{m_e}}=1. 
  \end{align}
  We calculate for $e=\{i,j\}$, using $(\overline\psi^tF_{\cdot e}F^t_{e\cdot}\psi)^2=0$ and 
  $(\overline\psi^tF_{\cdot e}m_eQ_eF^t_{e\cdot}\psi)^2=0$,

  \begin{align}
    S_e^{m_e}=&\left(B_e+(\overline\psi_i-\overline\psi_j)(\psi_i-\psi_j)e^{u_i+u_j}\right)^{m_e}
    =(B_e+e^{u_i+u_j}\overline\psi^tF_{\cdot e}F^t_{e\cdot}\psi)^{m_e}\nonumber\\
    =&B_e^{m_e}+m_eB_e^{m_e-1}e^{u_i+u_j}\overline\psi^tF_{\cdot e}F^t_{e\cdot}\psi
    =B_e^{m_e}\left(1+\frac{m_e}{B_e}e^{u_i+u_j}
     \overline\psi^tF_{\cdot e}F^t_{e\cdot}\psi \right)\nonumber\\
     =&B_e^{m_e}\left(1+\overline\psi^tF_{\cdot e}m_eQ_eF^t_{e\cdot}\psi \right)
        =B_e^{m_e}\exp\left(\overline\psi^tF_{\cdot e}m_eQ_eF^t_{e\cdot}\psi \right).
     \label{eq:expr-S-e-hoch-me}
\end{align}
Hence, taking a product over $e$ and using $\sum_{e\in E}F_{\cdot e}m_eQ_eF^t_{e\cdot}=FMQF^t$,
we obtain
\begin{align}
  \label{eq:id-1-equals-sk}
1=\sk{\prod_{e\in E}  S_e^{m_e}}
     =&\sk{\exp\left(\overline\psi^tFMQF^t\psi \right)\prod_{e\in E} B_e^{m_e}}.
\end{align}
The Grassmann part in the action $A$ in \eqref{eq:defA} is given by
\begin{align}
  \sum_{\{i,j\}\in E_+}W_{ij}
  (\overline\psi_i-\overline\psi_j)(\psi_i-\psi_j)e^{u_i+u_j}
  &=\sum_{e\in E_+}\overline\psi^tF_{\cdot e}\W_eF^t_{e\cdot}\psi\nonumber\\
  &=\overline\psi^tF\W F^t\psi
  =\overline\psi^tD\psi. 
\end{align}
Using the definition \eqref{eq:def-mu} of $\mu$, we can rewrite \eqref{eq:id-1-equals-sk} as
\begin{align}
  1=&\int\limits_{\R^\Lambda\times\R^\Lambda}\prod_{i\in \Lambda}\partial_{\overline\psi_i}\partial_{\psi_i}
      \exp(-\overline\psi^t(D-FM QF^t)\psi)     \prod_{e\in E}B_e^{m_e}
  \, \frac{\mu(du\, ds)}{\det D}\nonumber\\
  =&\E\left[\prod_{e\in E}B_e^{m_e}\frac{\det(D-FM QF^t)}{\det D}\right].
\end{align}
Using the identity $\det(\Id+AB)=\det(\Id+BA)$, which holds for arbitrary
rectangular matrices $A\in\R^{m\times n}$, $B\in\R^{n\times m}$, and the fact
$M Q=\sqrt{Q}M \sqrt{Q}$, we obtain 
\begin{align}
  \label{eq:rel-det-id-min2}
\frac{\det(D-FM QF^t)}{\det D}
  =& \det(\Id-D^{-1}FM QF^t)\nonumber\\
  =&\det(\Id-M \sqrt{Q}F^tD^{-1}F\sqrt{Q})
  =\det(\Id-M\G),
\end{align}
where in the second expression the identity matrix $\Id$ is indexed by vertices, while
in the third and fourth expression it is indexed by edges.
We conclude claim \eqref{eq:susy-with-det}.

Assume now $0\le m_e<W_e$ for $e\in E_+$ and $m_e=0$ otherwise. We
calculate
\begin{align}
  \label{eq:id-mG-detD}
  \det(\Id-M\G)\det D=& \det(D-FM QF^t)
                      =\det(F(\W-MQ)F^t). 
\end{align}
Note that $F(\W-MQ)F^t$ is a discrete Laplacian with weights
$\W_e-m_eQ_e$ rather than $\W_e$. Edges in $E_+$ have a
positive weight $\W_e-m_eQ_e=\W_e(1-\frac{m_e}{W_eB_e})>0$ since
$B_e\ge 1$. 
Using the matrix tree theorem and writing $\cS$ for the
set of spanning trees of the graph $G_+$, we rewrite the last determinant
as follows
\begin{align}
  \det(F(\W-MQ)F^t)=\sum_{T\in\cS}\prod_{e\in T}(\W_e-m_eQ_e)
  =\sum_{T\in\cS}\prod_{e\in T}\W_e\left(1-\frac{m_e}{W_eB_e}\right).
\end{align}
For each $T\in\cS$, using $B_e\ge 1$ again, we have 
\begin{align}
  \prod_{e\in T}\left(1-\frac{m_e}{W_eB_e}\right)
  \ge \prod_{e\in E_+}\left(1-\frac{m_e}{W_eB_e}\right)
  \ge \prod_{e\in E_+}\left(1-\frac{m_e}{W_e}\right).
\end{align}
Therefore, we obtain 
\begin{align}
  \det(F(\W-MQ)F^t)
  \ge \prod_{e\in E_+}\left(1-\frac{m_e}{W_e}\right) \sum_{T\in\cS}\prod_{e\in T}\W_e
  =\prod_{e\in E_+}\left(1-\frac{m_e}{W_e}\right) \det D.
\end{align}
It follows from \eqref{eq:id-mG-detD} that
\begin{align}
  \det(\Id-M\G)
  \ge \prod_{e\in E_+}\left(1-\frac{m_e}{W_e}\right).
  \end{align}
  Inserting this in \eqref{eq:susy-with-det} and using $\cosh (u_{e_{+}}-u_{e_{-}})\le B_e$, the claim
  \eqref{eq:bound-prod-Be-without-protection} follows. 
\end{proof}

\subsection{Proof of Theorem~\ref{th:est-cosh-euclidean-highdim}}
\label{se:boundshighd}

\begin{proof}[Proof of Theorem~\ref{th:est-cosh-euclidean-highdim}]
  Take $\wzwei$ large enough, to be specified below, and
  let $\overline W\ge \wzwei$. 
By assumption, $W_{ij}\geq \overline W 1_{\{\|i-j\|_{2}=1 \}}=: W^{\rm nn}_{ij}$
and $h_{i}=W_{i\rho }\geq  W^{\rm nn}_{i\rho }=:h^{\rm nn}_{i}$ where $ W^{\rm nn}_{i\rho }$ is defined as in
\eqref{eq:def-pinning-general}. Therefore, for $m\ge 1$,
by Fact~\ref{fact:poudevigne}, we can compare our model with weights $W,h$ with
the nearest neighbor model with weights $W^{\rm nn},h^{\rm nn}$:
\begin{align}
  \label{eq:inequ-cosh-eucl-highdim}
  \E^{\Lambda_{N}}_{W,h}\left[e^{\sigma m u_i}\right]\le
  \E^{\Lambda_{N}}_{W^{\rm nn},h^{\rm nn}}\left[e^{\sigma m u_i}\right]. 
\end{align}
Note that we have $h^{\rm nn}_i=0$ for $i\notin\partial \Lambda_N$, 
  while $h^{\rm nn}_i=W_{i\rho}>0$ for $i\in\partial \Lambda_{N}$.
Therefore, we take 
$j\in \partial \Lambda_{N}$ and use the Cauchy-Schwarz inequality to obtain
\begin{align}
  \label{eq:intermediate-bound-exp-sigma-m}
  \E^{\Lambda_{N}}_{W^{\rm nn},h^{\rm nn}}\left[e^{\sigma m u_i}\right]
&\leq 
 \E^{\Lambda_{N}}_{W^{\rm nn},h^{\rm nn}}\left[e^{2\sigma m(u_i-u_{j})}\right]^{\frac{1}{2}}
\E^{\Lambda_{N}}_{W^{\rm nn},h^{\rm nn}}\left[e^{2\sigma m u_j}\right]^{\frac{1}{2}}.
\end{align}
Assuming now that $\wzwei$ is so large that $\wzwei\ge 2^8$ and 
\cite[Theorem~1]{disertori-spencer-zirnbauer2010} is applicable, we have
\begin{align}
  \label{eq:bound-2-m-plus-1}
  \E^{\Lambda_{N}}_{W^{\rm nn},h^{\rm nn}}\left[e^{2\sigma m (u_i-u_{j})}\right]
  \le
2^{2m}\E^{\Lambda_{N}}_{W^{\rm nn},h^{\rm nn}}\left[(\cosh (u_i-u_{j}))^{2m}\right]\leq 2^{2m+1}
\end{align}
for all $m\in [1,\frac12\overline W^{\frac{1}{8}}]\neq\emptyset.$
Note that this result holds for any choice of the pinning
and in any dimension $d\geq 3.$\\
The next bound uses $u_\rho=0$, the estimate \eqref{eq:bound-prod-Be-without-protection}
from Lemma~\ref{le:prod-equals-1} which is valid on any finite graph, 
and the inequality $W_{j\rho }\geq \overline W$, which follows from
  $j\in \partial \Lambda_{N}$, to obtain
\begin{align}
\E^{\Lambda_{N}}_{W^{\rm nn},h^{\rm nn}}\left[e^{2\sigma m u_{j}}\right]
  \le 2^{2m} \E^{\Lambda_{N}}_{W^{\rm nn},h^{\rm nn}}\left[(\cosh u_j)^{2m}\right]
  &=2^{2m}\E^{\Lambda_{N}}_{W^{\rm nn},h^{\rm nn}}\left[(\cosh (u_{j}-u_{\rho }))^{2m}\right]\nonumber\\
  &\leq  
    \frac{2^{2m}}{1-\frac{2m}{W_{j\rho }}}\leq \frac{2^{2m}}{1-\frac{2m}{\overline W}}\leq
    2^{2m+1},
\end{align}
where the last step is a consequence of 
$2m\leq \overline W^{\frac{1}{8}}$ and $\overline W\geq 2^{8/7}$.
Substituting this and \eqref{eq:bound-2-m-plus-1} in \eqref{eq:intermediate-bound-exp-sigma-m},
the result \eqref{eq:eucl-est-cosh-ui2} now follows. 
\end{proof}

Note that it would suffice to do this proof for $d=3.$
Indeed  for $d\geq 4$ one can  embed $\Z^{3}$ in $\Z^{d}$
and use  Poudevigne's monotonicity (Fact~\ref{fact:poudevigne}) together with the result in $d=3.$

\paragraph{Acknowledgments.}
We are grateful to 
Malek Abdesselam for suggesting to compare the hierarchical model with the long-range
model on $\Z^d$.  We also thank Christophe Sabot, Pierre Tarr\`es, and Xiaolin Zeng 
for useful discussions. We would like to thank an anonymous associate editor
for suggesting a modification of the main results and a major simplification
of their proof. We would also like to thank an anonymous referee for
careful reading and constructive remarks. 

This work was supported by the DFG priority program SPP 2265 Random Geometric Systems.

\end{document}